\documentclass{amsart}
\usepackage{hyperref}
\usepackage[all]{xy}
\usepackage{mathabx}

\newtheorem{theorem}{Theorem}[section]
\newtheorem{lemma}[theorem]{Lemma}

\theoremstyle{definition}
\newtheorem{definition}[theorem]{Definition}

\theoremstyle{remark}
\newtheorem{remark}[theorem]{Remark}

\numberwithin{equation}{section}

\DeclareMathOperator{\modulo}{mod}
\DeclareMathOperator{\ima}{im}

\DeclareMathOperator{\add}{add}
\DeclareMathOperator{\Hom}{Hom}
\DeclareMathOperator{\coker}{coker}
\DeclareMathOperator{\ident}{id}
\DeclareMathOperator{\out}{out}
\DeclareMathOperator{\din}{in}
\DeclareMathOperator{\nilp}{nil}
\DeclareMathOperator{\module}{-mod}

\begin{document}

\title{Mutation of representations and nearly Morita Equivalence}

\author{ Diego Velasco}
\address{Instituto de Matem\'aticas, Ciudad Universitaria, 04510 M\'exico D.F., M\'exico}
\curraddr{}
\email{diego.velasco@matem.unam.mx}
\thanks{}



\date{October 2015}

\dedicatory{}

\begin{abstract}
In \cite{BIRS} it was proved, based on \cite{DWZI}, that the Jacobian algebra of two quivers with potential related by a QP-mutation are nearly Morita equivalent. They proved, using  Axiom of Choice, that the natural functor $\mu_k$ is an equivalence by showing that $\mu_k$ is full, faithfull and dense. In this note we provide a quasi-inverse $\mu_k^-$ to $\mu_k$ without Axiom of Choice.

\end{abstract}

\maketitle


\section{Introduction}
Let  $Q$ be an acyclic quiver, $K$ a field and $k\in Q_0$ a  sink of $Q$. In \cite{BGP}, Bernstein-Gelfand-Ponomarev  introduced a pair of adjoint functors $(F^-_k, F^+_k)$ called now BGP reflection functors or BGP-functors. We have
\begin{equation}\label{F+}
\xymatrix{  K\langle Q\rangle\module\ar@<0.5ex>[r]^{F^+_k}   &  K\langle Q'\rangle\module\ar@<0.5ex>[l]^{F^-_k} },
\end{equation}
where $Q'$ is obtained from $Q$ by changing the direction of all arrows incident to $k$. Here, $K\langle Q\rangle$ and $K\langle Q'\rangle$ are the path algebras of $Q$ and $Q'$ respectively. 
It was noted in \cite{APR} and \cite{DR} that these BGP-functors induce mutually quasi-inverse functors between the quotient categories $K\langle Q\rangle\module/[\add S_k]$ and $K\langle Q'\rangle\module/[\add S'_k]$ where $[\add S_k]$ is the ideal of morphisms that factorize through direct sums  of $S_k$, the simple $K\langle Q\rangle$-module at $k$.
Following Ringel, \cite{Rin}, we say in this case that $K\langle Q\rangle$ and $K\langle Q'\rangle$ are \textit{nearly Morita equivalent}. 

More generally one can consider, for a 2-acyclic quiver and any $k\in Q_0$ , the quiver mutation $\mu_k(Q)$, see for example \cite{FZI}. However, in this case we may not expect to relate $K\langle Q\rangle\module$ and $K\langle \mu_k(Q)\rangle\module$ in a meaningful way.

In \cite{DWZI}, Derksen, Weyman and Zelevinsky defined  mutations of  quivers with potential and also they defined  mutations of the representations of a quiver with potential. In \cite[Theorem 10.13]{DWZI} it was proved that the mutation of representations is an involution up to right-equivalence. 

 In \cite[Section 7]{BIRS}  a natural functor $\mu_k$ is constructed.   This functor $\mu_k$ is defined from  $\mathcal{P}(Q,S)$-$\modulo/[\add S_k]$ to $\mathcal{P}(\mu_k(Q,S))$-$\modulo/[\add S'_k] $, where $\mathcal{P}(Q,S)$ denotes the Jacobian algebra of a quiver with potential $(Q,S)$. This functor is based on the notion of mutation of representation introduced in \cite{DWZI} . By making use of \cite[Theorem 10.13]{DWZI} it was proved,  \cite[Theorem 7.1]{BIRS}, that $\mu_k$  is full, faithful and dense. It is well known that this implies, by the Axiom of Choice,  that the functor is an equivalence (see, for example \cite[Appendix A, 2.5]{ASS}).

Let us write from now on $\mu_k^+:=\mu_k$. In this note we produce explicitly a quasi-inverse $\mu_k^-$ of $\mu_k^+$, which is, in fact, quite similar to $\mu_k^+$. 

Since we are interested in the factor category $\mathcal{P}(Q,S)$-$\modulo/[\add S_k]$ we avoid to work with decorated representations. To give the quasi-inverse $\mu_k^-$ explicitly we proceed as follows. First at all, following \cite{DWZI}, for the convenience of the reader in Section \ref{section2} and Section \ref{section3} we recall some background on mutation of quivers with potential and the mutation of their representations.  The rest of this note is devoted    to define $\mu_k^+$, $\mu_k^-$, and show that they are quasi-inverses, see Theorem \ref{qinverse}.
\\ 

\textbf{Acknowledgements.} Thanks to Daniel Labardini-Fragoso and Christof Gei\ss \ for  suggesting  this subject and for helpful discussions.
\section{Background on quivers with potential}\label{section2}
A \textit{quiver}  $Q=(Q_0, Q_1, t, h)$ consists of a finite set of vertices $Q_0$,  a finite set of arrows $Q_1$ and two maps $t,h: Q_1\rightarrow Q_0$ (head, and tail). For each $a\in Q_1$we write $a:t(a)\rightarrow h(a)$. Given a  algebraically closed field $K$  we will denote by $R=\bigtimes_{i\in Q_0} K$ to the \textit{vertex space} and by $A=\bigtimes_{\alpha \in Q_1} K$ to the \textit{arrow space}.  We have $R$ is a semisimple $K$-algebra with the usual addition and multiplication defined coordinate-wise and  $A$ is an $R$-bimodule with following structure:
$$ (x_l)_{l\in Q_0} \cdot ((y_a)_{a\in Q_1})= (x_{h(a)}y_a)_{a\in Q_1}$$
$$((y_a)_{a\in  Q_1})\cdot (x_l)_{l\in Q_0}= (y_ax_{t(a)})_{a\in Q_1}.$$

 For $l>0$, let $A^l=A\otimes_R\cdots \otimes_R A$ be the $l$-fold tensor product over $R$ of $A$ with itself as $R$-bimodule. The \textit{path algebra} of $Q$ is defined as the tensor algebra $K\langle Q \rangle=\bigoplus_{l\geq 0}{A^l}$ and the \textit{complete path algebra} is defined as  the complete tensor algebra $K \langle\langle Q \rangle \rangle=\prod_{l\geq0}{A^l}$. 

We say that a sequence of arrows  $\alpha=a_la_{l-1}\cdots a_2a_1$, is a \textit{path} of $Q$ if  $t(a_{k+1})=h(a_{k})$, for $k=1, \ldots l-1$, in this case, we define  the \textit{length} of $\alpha$ as $l$. We say that $\alpha$ is a \textit{cycle} if $h(a_l)=t(a_1)$,  Then we can think  the elements of  $K\langle Q \rangle$ as  $K$-linear combinations of paths and the elements of $K \langle\langle Q \rangle \rangle$ as  possibly infinite $K$-linear combinations of paths. 

Let us recall some important fact about the complete path algebra  $K \langle\langle Q \rangle \rangle$. Let $\mathfrak{M}=\prod_{l\geq 1}{A^l}$ be the two-sided ideal of $K \langle\langle Q \rangle \rangle$ generated by arrows of $Q$. Then $K \langle\langle Q \rangle \rangle$ can be viewed as a \textit{topological} $K$-\textit{algebra} with the powers of $\mathfrak{M}$ as a basic system of open neighborhoods of $0$. This topology is known as $\mathfrak{M}$-\textit{adic topology}. Now, giving  $I\subseteq   K \langle\langle Q \rangle \rangle$ we can calculate the closure of $I$ as $\overline{I}=\bigcap_{l\geq0}{(I+\mathfrak{M}^l)}$.

We say, by a slight abuse of language, that a $K$-algebras homomorphism $\phi: K\langle\langle Q\rangle\rangle \rightarrow K\langle\langle Q'\rangle\rangle$  is a \textit{homomorphism of} $R$-\textit{algebras} if  $\phi(r)=r$  for each $r\in R$. If this is the case,  $\phi$ is continuous as morphism of topological algebras (see \cite[Section 2]{DWZI}).

A \textit{finite-dimensional representation} of $Q$ over $K$ is a pair $((M_i)_{i\in Q_0}, (M_a)_{a\in Q_1})$ where $M_i$ is a finite-dimensional vector space over $K$ for each $i\in Q_0$ and $M_a: M_{t(a)}\rightarrow M_{h(a)}$ is a $K$-linear map. Here the word \textit{representation} means finite dimensional representation. We say that $M$ is a  \textit{nilpotent} representation if there is an $n>0$ such that for every  path $a_na_{n-1}a\ldots a_1$ of length $n$ in $Q$  we have $M_{a_n}M_{a_{n-1}}\ldots M_{a_1}=0$.

 We denote by $\nilp_K(Q)$ the category of nilpotent representations of $Q$, and by $K\langle\langle Q\rangle\rangle$-$\modulo$ the category of finite-dimensional left $K\langle\langle Q\rangle\rangle$-modules. It is well known that the categoriy of representations of $Q$ and the category of $K\langle Q \rangle$-modules are equivalent. In \cite[Section 10]{DWZI} it was observed  that  $\nilp_K(Q)$ and $K\langle\langle Q\rangle\rangle$-$\modulo$ are equivalent.

\subsection{Quivers with potential and their mutations}
In this preliminary section, for the convenience of the reader, we shall  recall some basic definitions and facts from \cite{DWZI}.
Let $Q$ be  a quiver. We say that $S\in K\langle\langle Q\rangle\rangle$ is a \textit{potential} for $Q$ if $S$ is a, possibly infinite, $K$-linear combination  of cycles in $Q$ . Given two potentials $S$ and $W$  we say that they are \textit{cyclically equivalent} and write $S\sim_{\operatorname{cyc}} W$, if $S-W$ is in the closure of the sub-vector space of $K\langle\langle Q\rangle\rangle$   generated by all elements of the form $a_1a_2\cdots a_{n-1}a_n -a_2\cdots a_{n-1}a_na_1$, with  $a_1a_2\cdots a_{n-1}a_n$ a cycle on $Q$.

\begin{definition}
We say $(Q,S)$ is a \textit{quiver with potential} (QP) if $Q$ does not have loops, $S$ is a potential for $Q$ and if any two different cycles appearing with non-zero coefficient   in $S$ are not  cyclically equivalent.
\end{definition}

Given $a\in Q_1$ and a cycle $a_na_{n-1}\cdots a_1$ in $Q$, define the \textit{cyclic derivative} of $a_na_{n-1}\cdots a_1$ with respect to $a$ as follows:
$$\partial_{a}(a_na_{n-1}\cdots a_1)=\sum\limits_{k=1}^{n}{\delta_{a, a_k}a_{k-1}a_{k-2}\cdots a_1a_na_{n-1}\cdots a_{k+2}a_{k+1}}.$$
We extend this definition by $K$-linearity and continuity to all potentials for $Q$.
\begin{definition}
Let $(Q,S)$ be a quiver with potential. We define the \textit{Jacobian ideal}  $\mathcal{J}(Q,S)$ as the closure of the ideal on $K\langle\langle Q\rangle\rangle$ generated by  all cyclic derivatives $\partial_a(S)$ with $a\in Q_1$. The quotient $ K\langle\langle Q\rangle\rangle/J(Q,S)$ is called the \textit{Jacobian algebra} of $(Q,S)$ and is denoted as $\mathcal{P}(Q,S)$.
\end{definition}
A quiver with potential $(Q,S)$ is \textit{trivial} if the $R$-$R$-bimodule on $k\langle\langle Q\rangle\rangle$ generated by $\partial_a(S)$ is $A$ with $a\in Q_1$. We call $(Q,S)$ \textit{reduced} if $S$ does not have cycles of length 2. 

\begin{definition}
Let $(Q,S)$ and $(Q',S')$ be quivers with potential. We say $\phi: K\langle\langle Q\rangle\rangle \rightarrow K\langle\langle Q'\rangle\rangle$ is a \textit{right-equivalence} of QP if $\phi$ is a $R$-algebra isomorphism and $\phi(S)\sim_{\operatorname{cyc}}S'$. 
\end{definition}

In \cite[Theorem 4.6]{DWZI} it was proved that given a quiver with potential $(Q,S)$,  there are a trivial quiver with potential $(Q_{\operatorname{triv}},S_{\operatorname{triv}})$, a reduced quiver with potential $(Q_{\operatorname{red}}, S_{\operatorname{red}})$ and a right-equivalence $\phi$ such that $\phi(S)\sim_{\operatorname{cyc}}S_{\operatorname{triv}}+S_{\operatorname{red}}$. This result is called \textquotedblleft the splitting Theorem\textquotedblright.

Let us recall  the notion of quivers mutation. Let $Q$ be a quiver and $k\in Q_0$ a vertex, if $Q$ does not have any cycle of length 2 (2-cycle) based at $k$, the mutation of $Q$ with respect to $k$, denoted $\mu_k(Q)$, can be obtained by  the next three steps.

\begin{enumerate}
\item For each pair of arrows $a:j\rightarrow k$ and $b:k\rightarrow i$, add a new arrow $[ba]:j\rightarrow i$.\\
\item Replace each arrow $a: j\rightarrow k$ with  $a^*: k\rightarrow j$ and $b:k \rightarrow i$ with $b^*:i\rightarrow k$.\\
\item Remove any maximal disjoint collection of oriented 2-cycles.
\end{enumerate}

Now we define QP-mutation, following \cite{DWZI}. Let $(Q,S)$ be a quiver with potential. Suppose that $Q$ is 2-acyclic and there are no cycles in $S$ that begin at $k$. We define the \textit{mutation} of $(Q,S)$ with respect to  $k$ as $\mu_k(Q,S)= (\widetilde{\mu}_k(Q)_{\operatorname{red}}, \widetilde{S}_{\operatorname{red}})$, where $\widetilde{\mu}_k(Q)$ is  obtained from $Q$ by applying the two first steps of $k^{\operatorname{th}}$-mutation $\mu_k$ and $\widetilde{S}$ is:
\begin{equation}\label{def widetilde{S}}
  \widetilde{S}=[S]+\sum\limits_{a,b\in Q_1: h(a)=k=t(b) }{[ba]a^*b^*},
\end{equation}
where $[S]$ is obtained  from $S$ by replacing any pair of arrows $a:j\longrightarrow k$ and $b:k\longrightarrow i$ in the expansion of $S$ with $[ba]$. In \cite{DWZI} is showed that $\mu_k(Q,S)$ is well defined up to right equivalence. 
\subsection{Mutation of representations}

Let $(Q,S)$ be a quiver with potential. We say that $M$ is a \textit{representation} of $(Q,S)$ if $M$ is a finite-dimensional representation of $Q$  and $M$ satisfies the cyclic derivatives $\partial_a(S)$ for each $a \in Q_1$. Now we extend the notion of right-equivalence from quivers with potential to their representations. Our definition is a bit different from \cite[Definition 10.2]{DWZI}. 

\begin{definition}\label{def r-equivalence}
Let $M$ and $M'$ be representations of $(Q,S)$ and $(Q',S')$. A pair $(\phi, \psi)$ is a \textit{right-equivalence} from $M$ to $M'$ if it satisfies the following:
\begin{enumerate}
\item $\phi:K\langle\langle Q\rangle\rangle \rightarrow K\langle\langle Q'\rangle\rangle$ is a right-equivalence between $(Q,S)$ and $(Q',S')$.\\
\item $\psi: M\rightarrow M'$ is a vector space isomorphism such that $\psi\circ (u\cdot -)= (\phi(u)\cdot -)\circ \psi$, for every $u\in K\langle\langle Q\rangle\rangle$.
\end{enumerate}
\end{definition}

\begin{remark}\label{M tor}
Denote $\Lambda=K\langle\langle Q\rangle\rangle$. Note that if we define a $\Lambda$-module $M'^{\phi}$ as $M'$ like vector space  with the structure $u\cdot m'=\phi(u)\cdot m'$ the second condition of Definition \ref{def r-equivalence} can be restated by saying that $\psi:M\rightarrow M'^{\phi}$ is a  $\Lambda$-module isomorphism.
\end{remark}
 
\begin{definition}\label{redPart}
Let  $$\phi:K\langle\langle Q_{\operatorname{triv}}\bigoplus Q_{\operatorname{red}}\rangle\rangle\rightarrow K\langle\langle Q\rangle\rangle$$ be a right-equivalence. 
 We define the \textit{reduced part} of $M$, denoted $M_{\operatorname{red}}$, as $M^{\phi_{r}}$  (see  Remark \ref{M tor}), where $\phi_{r}$ is the restriction of $\phi$ to $K\langle\langle Q_{\operatorname{red}}\rangle\rangle$.
\end{definition}

Let $(Q,S)$ be a quiver with potential and suppose that $Q$ is 2-acyclic. Mutation of representations of $(Q,S)$ was introduced in \cite{DWZI}. Let $M$ be a representation of $(Q,S)$ and   $k$ a vertex in $Q$. 
Set $\{a_1, a_2\ldots , a_s\}=\{a\in Q_1| \  h(a)=k\}$ with $a_i\neq a_j$ for $i\neq j$ and $\{b_1, b_2, \ldots b_t\}=\{b\in Q_1| \ t(b)=k\}$ with $b_i\neq b_j$ for $i\neq j$. Define

\begin{eqnarray}\label{M_in M_out}
M_{\din}(k)=\bigoplus \limits_{p=1}^{s}{M_{t(a_p)}},  \    \    \    \      M_{\out}(k)=\bigoplus \limits_{q=1}^{t}{M_{h(b_q)}}
\end{eqnarray}

We get  natural linear maps $\alpha_{k, M}:M_{\din}(k)\longrightarrow M_k$ and $\beta_{k, M}:M_k\longrightarrow M_{\out}(k)$, in matrix form
\begin{eqnarray}\label{alpha, beta}
\alpha_{k,M}=(a_1 \ldots a_s),  \  \   \    \ 
\beta_{k, M}=
\left(
\begin{array}{lcr}
b_{1}  \\
b_{2}  \\
\vdots\\
b_{t}  
\end{array}
\right)
\end{eqnarray}

Let $c_1\cdots c_n$ be a cycle in $Q$, we define 

\begin{equation}
\partial_{b_qa_p}(c_1\cdots c_n)= \sum\limits_{j=1}^{n}{\delta_{b_qa_p,c_jc_{j+1}}c_{j+2}\cdots c_nc_1\cdots c_{j-1}},
\end{equation}
for $p=1\ldots s$ and $q=1\ldots t$. We extend this definition by $K$-linearity and continuity to all potentials for $Q$. 

Define the linear map $\gamma_{k,M}:M_{\out}(k)\longrightarrow M_{\din}(k)$ in matrix form

\begin{equation}\label{def gamma}
 (\gamma_{k,M})_{p,q}=\partial_{b_qa_p}(S):M_{h(b_q)}\longrightarrow M_{t(a_p)}
\end{equation}
  
It  is useful to keep in mind the following \textit{local triangle associated} to $M$ at $k$, which summarizes  the data that we got so far:

\begin{equation} 
\xymatrix{ \  & M_{k} \ar@{->}[dr]^{\beta_{k,M}} &  \\
           M_{\din}(k) \ar@{->}[ru]^{\alpha_{k,M}}  & \ \  &  M_{\out}(k) \ar@{->}[ll]^{\gamma_{k,M}} }
\end{equation}
	
We write $(\widetilde{Q},\widetilde{S})=\widetilde{\mu}_k(Q,S)$. In what follows we define the  \textit{pre-mutation}  $\widetilde{\mu}_k(M)$ as a representation of $(\widetilde{Q},\widetilde{S})$, for short $\overline{M}:=\widetilde{\mu}_k(M)$. As vector space, $\overline{M}_i=M_i$, if $i\neq k$, and if $i=k$ we set

\begin{equation}\label{def _Mk}
\overline{M}_k=\frac{\ker (\gamma_{k,M})}{\ima(\beta_{k,M})}\oplus \ima(\gamma_{k,M})\oplus 
           \frac{\ker(\alpha_{k,M})}{\ima(\gamma_{k,M})}.         
\end{equation}                  

For the action of the arrows we note that $\overline{M}_{\din}(k)=M_{\out}(k)$ and $\overline{M}_{\out}(k)=M_{\din}(k)$, now define

\begin{equation}\label{alpha* beta*}
\alpha_{k,\overline{M}}=(b^*_1, \ldots b^*_t),  \  \   \    \ 
\beta_{k, \overline{M}}=
\left(
\begin{array}{lcr}
a_{1}^*  \\
a_{2}^*  \\
\vdots\\
a_{s}^*  
\end{array}
\right).
\end{equation}

We choose a retraction and a section
\begin{equation}\label{rho, sigma}
\rho_M:M_{\out}(k)\longrightarrow \ker(\gamma_{k,M}), \ \ \ \sigma_M:\ker(\alpha_{k,M})/\ima(\gamma_{k,M})\longrightarrow \ker(\alpha_{k,M})
\end{equation}
In other words, we have $\rho_M\iota=\ident_{\ker(\gamma_{k,M})}$ and $\pi\sigma_M=\ident_{\ker(\alpha_{k,M})/\ima(\gamma_{k,M})}$, trough out $\iota$ and $\pi$ are the natural inclusion and projection respectively. We are ready to define the action of $\widetilde{Q}$ in $M$

\begin{equation}\label{def alpha* beta*}
\alpha_{k,\overline{M}}=
\left(
\begin{array}{lcr}
-\pi \rho_M  \\
-\gamma_{k,M} \\
 \  \ 0\\
\end{array}
\right),  \  \   \    \ 
\beta_{k, \overline{M}}=(0, \iota, \iota \sigma_M).
\end{equation}	
	In \cite{DWZI} the  \textit{mutation} in direction $k$ is defined   as $\widetilde{\mu}_k(M)_{\mbox{red}}=\overline{M}_{\mbox{red}}$, see Definition \ref{redPart}. There, the authors  proved that $\overline{M}$   is actually  a representation of $\widetilde{\mu}_k(Q,S)$ and that it does not depend of the splitting data \eqref{rho, sigma}, up to isomorphism, see \cite[Proposition 10.9]{DWZI}. Another fact proved in the same work is that the class of right-equivalence of $\widetilde{\mu}_k(M)_{\mbox{red}}$ is determined by the class of right-equivalence of $M$, see \cite[Proposition 10.10]{DWZI}.

	
\section{Involutivity of mutation}\label{section3}

Let $(Q,S)$ be a reduced  QP with no 2-cycles at $k$, and $M$ be a representation of $(Q,S)$. Denote by $\overline{\overline{M}}$  the representation $\widetilde{\mu}_k(\widetilde{\mu}_k(M))$ of $(\widetilde{\widetilde{Q}}, \widetilde{\widetilde{S}})=\widetilde{\mu}_k(\widetilde{\mu}_k(Q,S))$. From \cite[Theorem 5.7]{DWZI}, \cite[Proposition 10.11]{DWZI} and Remark \ref{M tor} we can see that $\mu_{k}^2(M)=\overline{\overline{M}}_{\operatorname{red}}=\overline{\overline{M}}^{\varphi}$, where $\varphi$ is defined as follows

\begin{equation}\label{right equiv}
\varphi: b_q\longmapsto -b_q, \mbox{ for $q=1, \ldots t$ and} \ \varphi \ \mbox{fix the rest of the arrows in} \ \widetilde{\widetilde{Q}}.
\end{equation}
 Where $\varphi$ is an automorphism of $K\langle\langle\widetilde{\widetilde{Q}}\rangle\rangle$ and  we are identifying the arrows $b_q$ on $Q$ with the arrows $b_q^{*^{*}}$ on $\widetilde{\widetilde{Q}}$ for $q= 1, \ldots, t$; see the proof of \cite[Theorem 5.7]{DWZI}.
 
In \cite[Theorem 10.13]{DWZI} was proved that $M$ is right-equivalent to $\mu_k^2(M)$. In fact, we can deduce from the proof of \cite[Theorem 10.13]{DWZI} a slightly sharper statement.
\begin{lemma}\label{involution}
Let $(Q,S)$ be a quiver with potential and $M$ be a representation of $(Q,P)$ such that $M$ does not have direct summands isomorphic to $S_k$. If $\varphi$ is as in \eqref{right equiv}, then $M$ is isomorphic to  $\overline{\overline{M}}^{\varphi}$.   
\end{lemma}

\section{Definition of the functor $\mu_k^+$}

Giving a reduced quiver with potential $(Q,S)$, let $\mathcal{P}(Q,S)$ be the Jacobian algebra of $(Q,S)$. We write $\Lambda=\mathcal{P}(Q,S)$ and  $\Lambda'=\mathcal{P}(\mu_k(Q,S))$.

 We fix a vertex $k\in Q_0$. If $f:M\rightarrow N$ be a morphism of representations of  $(Q,S)$, we say that $f$ is \textit{confined to} $k$ if  $f(m)=0$, $\forall m\in M_{\widehat{k}}$, with $M_{\widehat{k}}=\bigoplus_{i\neq k}{M_i}$. The set of all morphisms confined to $k$ is denoted by $\Hom_{\Lambda\modulo}^k(M,N)$.
 From \cite[Section 6]{DWZII} we have that   
\begin{equation*}
\Hom_{\Lambda\module / [\add  S_k]}(M,N) \ = \ \Hom_{\Lambda\module}(M,N)/\Hom_{\Lambda\module}^k(M,N).  
 \end{equation*}

In \cite[Theorem 7.1]{BIRS} and  \cite[Proposition 6.2]{DWZII} were proved that there is an equivalence between $\Lambda\module / [\add  S_k]$ and $\Lambda'\module / [\add  S'_k]$. Their proofs use the Axiom of Choice, indeed  their arguments consist by showing that a functor $F$ is full, faithful and dense, but it is well known that  $F$ is an equivalence thanks to the Axiom of Choice. In this note we provide an explicit quasi-inverse of the equivalence defined in \cite[Theorem 7.1]{BIRS}.

From \cite{BIRS} the functor $\mu_k^+: \Lambda\module / [\add  S_k]\rightarrow \Lambda'\module / [\add  S'_k]$ defined below is an equivalence. If $M\in \Lambda\module$, then $\mu_k^+(M)$ is defined as \eqref{def _Mk}. Now, given $f\in\Hom_{\Lambda\module}(M,N)$  we proceed as follows to define $\mu_k(f)$. We set $\mu_k(f)_j:=f_j$ for $j\neq k$. In order to define   $\mu_{k}^+(f)_k:\overline{M}_k\rightarrow \overline{N}_k$ we consider the following diagram obtained from \eqref{def _Mk}

\begin{equation}\label{exacta2}
\def\objectstyle{\scriptstyle}
\def\labelstyle{\scriptstyle}
\vcenter{\xymatrix @-0.8pc {
  0\ar@{->}[r]& \frac{\ker(\gamma_{k,M})}{\ima(\beta_{k,M})}\ar@<0.5ex>[r]^{\widetilde{i}_M}&
\coker(\beta_{k,M})\ar@{->}[rr]^{\iota_M \widetilde{\gamma}_{k,M}}\ar@<0.5ex>[dr] ^{\widetilde{\gamma}_{k,M}}\ar@<0.5ex>[l]^{\widetilde{\rho}_M}&   \  \  & \ker(\alpha_{k,M})\ar@<0.5ex>[r]^{\pi}\ar@<0.5ex>[ld]^{\epsilon_M}& \frac{\ker(\alpha_{k,M})}{\ima(\gamma_{k,M})}\ar@{->}[r]\ar@<0.5ex>[l]^{\sigma_M}& 0\\
 &  &    &  \ima(\gamma_{k,M})\ar@<0.5ex>[ru]^{\iota_M}\ar@{->}[dr]\ar@<0.5ex>[ul]^{j_M} &   &  & \\
  &    &   0\ar@{->}[ru]  &  & 0  &   & 
}}
\end{equation}
 
With $\widetilde{i}_M\colon\frac{\ker(\gamma_{k,M})}{\ima(\beta_{k,M})}\rightarrow \coker(\beta_{k,M})$ and  $\widetilde{\gamma}_{k,M}:\coker(\beta_{k,M})\rightarrow \ima(\gamma)$ the natural induced maps. The map $\widetilde{\rho}_M\colon \coker(\beta_{k,M})\rightarrow \frac{\ker(\gamma_{k,M})}{\ima(\beta_{k,M})}$ is induced by $\rho_{M}$ (see \eqref{rho, sigma}).
The choice of $\widetilde{\rho}_M$ and\ $\sigma_M$ allow to define maps $j_M$ and $\epsilon_M$ respectively, such that $\ident_{\coker(\beta_{k,M})}=\widetilde{i}_M\widetilde{\rho}_M+j_M \widetilde{\gamma}_{k,M}$ and $\ident_{\ker(\alpha_{k,M})}=\iota_M\epsilon_M +\sigma_M\pi$.
 
We get a similar diagram for $N$, then we obtain the following diagram where the central square is commutative 

\begin{equation}\label{diagram f_out}
\xymatrix{\frac{\ker(\alpha_{k,M})}{\ima(\gamma_{k,M})}\ar@{->}[dr]^{\sigma_M}& & & \frac{\ker(\alpha_{k,N})}{\ima(\gamma_{k,N})}\\
\bigoplus & \ker(\alpha_{k,M})\ar@{->}[r]^{f_{\din}}&\ker(\alpha_{k,N})\ar@{->}[ur]^{\pi'}\ar@{->}[dr]^{\epsilon_N}& \bigoplus\\
\ima(\gamma_{k,M})\ar@{->}[ur]^{\iota_M}\ar@{->}[dr]_{j_M} &  &  & \ima(\gamma_{k,N})\\
\bigoplus&\coker(\beta_{k,M})\ar@{-->}[uu]_{\widetilde{\gamma}_{k,M}}\ar@{->}[r]_{\overline{f_{\out}}}&\coker(\beta_{k,N})\ar@{->}[ur]^{\widetilde{\gamma}_{k,N}}\ar@{->}[dr]_{\widetilde{\rho}_N}\ar@{-->}[uu]_{\widetilde{\gamma}_{k,N}}&\bigoplus\\
\frac{\ker(\gamma_{k,M})}{\ima(\beta_{k,M})}\ar@{->}[ur]_{\widetilde{i}_M} &   &  &  \frac{\ker(\gamma_{k,N})}{\ima(\beta_{k,N})}
}
\end{equation}

Here $\overline{f_{\out}}$ is  the induced map by $f_{\out}\colon M_{\out}(k)\rightarrow N_{\out}(k)$. Now we can already give the definition of $\mu_k^+(f)_k$:

\begin{eqnarray}\label{f_k1}
\mu_k^+(f)_k=
\left(
\begin{array}{lcr}
\widetilde{\rho}_N\overline{f_{\out}}\widetilde{i}_M \ \ & \  \ \widetilde{\rho}_N\overline{f_{\out}}j_M & 0 \  \ \  \\
 \widetilde{\gamma}_{k,N}\overline{f_{\out}}\widetilde{i}_M \ \ & \ \ \epsilon_Nf_{\din}\iota_M \ \ & \ \ \epsilon_Nf_{\din}\sigma_M\\
\ \ \  0 \  \  & \  \  \pi' f_{\din}\iota_M \ \ & \  \ \pi'f_{\din}\sigma_M
\end{array}
\right).
\end{eqnarray}


\section{Quasi-inverse}

 The notion of  right-equivalence of representations, see Definition \ref{def r-equivalence}, is a delicate point in this subject. One reason is that  the classes of right-equivalence may have a different behavior of  isomorphism classes, for example  see \cite{DWZI}[Remark 10.3]. Indeed, it is possible to find two no isomorphic representations that are right-equivalent. Since \cite{DWZI}[Theorem 10.13] is stated in terms of the notion of right-equivalence, it is not obvious that one may say something about  isomorphism classes as in the proof of  \cite{BIRS}[Theorem 7.1]. In this section we use the  notion of right-equivalence for representations of a quiver with potential to define  explicitly a quasi inverse for $\mu_k^+$.

Let $(P,T)$ be a reduced  quiver with potential with no 2-cycles at $k\in P_0$. We define $\widetilde{\mu}_k^-(P,T)$ as $\widetilde{\mu}_k^-(P,T)=(\widetilde{\mu}_k(P), \widetilde{T}^-)$, with

\begin{equation}\label{S-}
  \widetilde{T}^{-}=[T]-\sum_{\substack{a,b\in P_1:\\
	h(a)=k=t(b)}}
	{[ba]a^*b^*}
\end{equation}

 We define $\mu_k^-(P,T)$ as the reduced part of $\widetilde{\mu}_k^-(P,T)$.  This notion of  mutation induces a mutation  of representations of $(P,T)$. Let $M$ be a representation of $(P,T)$. We define  $\widetilde{\mu}_k^-(M)$ as vector space the same form that $\widetilde{\mu}_k^+(M)$, (see \eqref{def _Mk}). We need define the actions of arrows in $\widetilde{\mu}_k^-(M)$, (see \eqref{alpha* beta*}). After choose maps as \eqref{rho, sigma} we obtain the new version of \eqref{def alpha* beta*} 

\begin{equation}\label{alpha- beta-}
\alpha_{k,\overline{M}^-}^-=
\left(
\begin{array}{lcr}
\pi \rho_M  \\
\gamma_{k,M} \\
 \  \ 0\\
\end{array}
\right),  \  \   \    \ 
\beta_{k, \overline{M}^-}^-=(0, \iota, \iota \sigma_M)
\end{equation}
Then we define $\mu_k^-(M)=\widetilde{\mu}_k^-(M)_{\mbox{red}}$. This is well defined up to right-equivalence.

Returning to our task, we have  $\mu_k^-:\Lambda'\modulo / [\add  S'_k]\rightarrow \Lambda\modulo / [\add  S_k]$, the notation is the same to the last section.

\begin{theorem}\label{qinverse}
$\mu_k^-$ is a quasi-inverse of $\mu_k^+$.
\end{theorem}

\begin{proof}
Let us write $M'$ for  $\mu_k^-\mu_k^+(M)$  to relax the notation. First, we  compute $M'$ as vector space. We then describe the  structure of $M'$ as $K\langle\langle Q\rangle\rangle$-module. 
 
By definition we have  $M'_i=M_i$, $i\neq k$. Now from \eqref{def _Mk} we can deduce that

\begin{equation}\label{def 'Mk}
M'_k=\frac{\ker (\gamma_{k,\overline{M}})}{\ima(\beta_{k,\overline{M}})}\oplus \ima(\gamma_{k,\overline{M}})\oplus 
           \frac{\ker(\alpha_{k,\overline{M}})}{\ima(\gamma_{k,\overline{M}})},
\end{equation}

where $\overline{M}=\widetilde{\mu}_k^+(M)$ as representation of $\widetilde{\mu}_k^+(Q,S)$. Applying the definitions we see
\begin{eqnarray}\label{rel alpha _alpha}
\ker(\alpha_{k,\overline{M}})=\ima(\beta_{k,M}), \  \  \ima(\alpha_{k,\overline{M}})=\frac{\ker(\gamma_{k,M})}{\ima(\beta_{k, M})}\oplus \ima(\gamma_{k,M})\oplus \left\{0\right\}, \\
\nonumber \ker(\beta_{k,\overline{M}})=\frac{\ker(\gamma_{k,M})}{\ima(\beta_{k, M})}\oplus \left\{0\right\}\oplus \left\{0\right\}, \   \  \ima(\beta_{k, \overline{M}})=\ker(\alpha_{k,M}),\\
\nonumber \ker(\gamma_{k, \overline{M}})=\ker(\beta_{k,M}\alpha_{k,M}) \  \ \ima(\gamma_{k, \overline{M}})=\ima(\beta_{k,M}\alpha_{k,M}).
\end{eqnarray}
By rewriting \eqref{def 'Mk} we get
\begin{equation}\label{def M'_k}
M'_k=\frac{\ker (\beta_{k,M}\alpha_{k,M})}{\ker(\alpha_{k,M})}\oplus \ima(\beta_{k,M}\alpha_{k,M})\oplus 
           \frac{\ima(\beta_{k,M})}{\ima(\beta_{k,M}\alpha_{k,M})}
\end{equation}
Now we observe:
\begin{itemize}
\item $\alpha_{k,M}$ induces an isomorphism $$\widetilde{\alpha}:\frac{\ker(\beta_{k,M}\alpha_{k,M})}{\ker(\alpha_{k,M})}\longrightarrow \ker(\beta_{k.M}), \  [x]\mapsto \alpha_{k,M}(x).$$

\item $\beta_{k,M}$ induces an isomorphism
$$\widetilde{\beta}:\frac{\ima(\alpha_{k,M})}{\ker(\beta_{k,M})}\longrightarrow \ima(\beta_{k,M}\alpha_{k,M}), \  [x]\mapsto \beta_{k,M}(x).$$

\item $\beta_{k,M}$ induces an isomorphism
$$\widehat{\beta}:\frac{M_k}{\ima(\alpha_{k,M})}\longrightarrow  \frac{\ima(\beta_{k,M})}{\ima(\beta_{k,M}\alpha_{k,M})}, \  [x]\mapsto [\beta_{k,M}(x)].$$
\end{itemize}

Since $M$ does not have direct summands isomorphic to $S_k$ we have $\ker(\beta_{k, M})\subseteq \ima(\alpha_{k,M})$. Combine \eqref{rel alpha _alpha} with the induced isomorphisms above to represent $M'_k$ in the following way

\begin{equation}\label{ident M'}
M'_k=\ker(\beta_{k.M})\oplus \frac{\ima(\alpha_{k,M})}{\ker(\beta_{k,M})}\oplus \frac{M_k}{\ima(\alpha_{k,M})}.
\end{equation}
With this form for $M'_k$ we choose linear maps as in \eqref{rho, sigma}, that is
$$\rho_{\overline{M}}:M_{\din}(k)\rightarrow \ker(\beta_{k,M}\alpha_{k,M}), \ \ \sigma_{\overline{M}}:\ima(\beta_{k,M})/\ima(\beta_{k,M}\alpha_{k,M})\rightarrow \ima(\beta_{k,M}).$$
Thus, $\rho_{\overline{M}}\iota=\ident_{\ker(\beta_{k,M}\alpha_{k,M})} $, and $\pi\sigma_{\overline{M}}=\ident_{\ima(\beta_{k,M})/\ima(\beta_{k,M}\alpha_{k,M})}$.
Now we  can define the structure of $M'$ as a representation. To do that, define linear maps as in \eqref{def alpha* beta*}:

\begin{equation}\label{def alpha' beta'}
\alpha_{k,M'}^-=
\left(
\begin{array}{lcr}
\alpha_{k,M} \rho_{\overline{M}}  \\
\pi\alpha_{k,M} \\
 \   \   \  \ 0
\end{array}
\right),  \  \   \    \ 
\beta_{k, M'}^-=(0, \widetilde{\beta}_{k,M}, \iota\sigma_{\overline{M}}\widehat{\beta}_{k,M}).
\end{equation}
Note that in  this case we do not need $\varphi$, see \eqref{right equiv} and the proof of \cite[Theorem 5.7]{DWZI}.

Let $f:M\rightarrow N$ be a morphism of representations of $(Q,S)$. With the form of $\mu_k⁻\mu_k^+(M)$ we can deduce the new version of \eqref{exacta2} and \eqref{diagram f_out}. Again the center square is commutative and $\overline{f_{\din}}$ is induced by $f_{\din}$.
\begin{equation}\label{diagram f_in}
\xymatrix{\frac{M_k}{\ima(\alpha_{k,M})}\ar@{->}[dr]^{\sigma_{\overline{M}}\widehat{\beta}_{k,M}}& & & \frac{N_k}{\ima(\alpha_{k,N})}\\
\bigoplus & \ima(\beta_{k,M})\ar@{->}[r]^{f_{\out}}&\ima(\beta_{k,N})\ar@{->}[ur]^{\widehat{\beta}_{k,N}^{-1}\pi}\ar@{->}[dr]^{\widetilde{\beta}_{k,N}^{-1}\epsilon_{\overline{N}}}& \bigoplus\\
\frac{\ima(\alpha_{k,M})}{\ker(\beta_{k,M})}\ar@{->}[ur]^{\iota_{\overline{M}}\widetilde{\beta}_{k,M}}\ar@{->}[dr]_{j_{\overline{M}}\widetilde{\beta}_{k,M}} &  &  & \frac{\ima(\alpha_{k,N})}{\ker(\beta_{k,N})}\\
\bigoplus&\frac{M_{\din}(k)}{\ker(\alpha_{k,M})}\ar@{-->}[uu]_{\widetilde{\gamma}_{k,\overline{M}}}\ar@{->}[r]_{\overline{f_{\din}}}&\frac{N_{\din}(k)}{\ker(\alpha_{k,N})}\ar@{->}[ur]_{\widetilde{\beta}_{k,M}^{-1}\widetilde{\gamma}_{k,\overline{N}}}\ar@{->}[dr]_{\widetilde{\alpha}_{k,N}\widetilde{\rho}_{\overline{N}}}\ar@{-->}[uu]^{\widetilde{\gamma}_{k,\overline{N}}}&\bigoplus\\ 
\ker(\beta_{k,M})\ar@{->}[ur]_{\widetilde{i}_{\overline{M}}\widetilde{\alpha}_{k,M}^{-1}} &   &  &  \ker(\beta_{k,N})
}
\end{equation}

Finally we define $\mu_k^-\mu_k^+(f)_k$ and denote it as $f'_k$:

\begin{eqnarray*}\label{def +-f}
f'_k=
\left(
\begin{array}{lcr}
\widetilde{\alpha}_{k,N}
\widetilde{\rho}_{\overline{N}}\overline{f_{\din}}\widetilde{i}_{\overline{M}}
\widetilde{\alpha}_{k,M}^{-1} \  &   \ 
\widetilde{\alpha}_{k,N}\widetilde{\rho}_{\overline{N}}\overline{f_{\din}}j_{\overline{M}} \widetilde{\beta}_{k,M}&   0 \  \  \  \  \  \  \  \  \  \  \\
\widetilde{\beta}_{k,N}^{-1}\widetilde{\gamma}_{k,\overline{N}}\overline{f_{\din}}\widetilde{i}_{\overline{M}}\widetilde{\alpha}_{k,M}^{-1} \ &  \ \widetilde{\beta}_{k,N}^{-1}\epsilon_{\overline{N}}f_{\out}\iota_{\overline{M}}
\widetilde{\beta}_{k,M} \ &  \ \widetilde{\beta}_{k,M}^{-1}\epsilon_{\overline{N}}f_{\out}\sigma_{\overline{M}}
\widehat{\beta}_{k,M}\\
\ \ \ \  \  \  \   \   \  \  \ 0 \   &   \  \widehat{\beta}_{k,N}^{-1}\pi f_{\out}\iota_{\overline{M}}\widetilde{\beta}_{k,M} \  &   \ \widetilde{\beta}_{k,N}^{-1}\pi f_{\out}\sigma_{\overline{M}}\widehat{\beta}_{k,M}
\end{array}
\right),
\end{eqnarray*}
remember that $f'_i=f_i$ if $i\neq k$.

To finish the proof we need to give an natural isomorphism $\psi_{k,M}: M'_k\rightarrow M_k$ such that 
\begin{equation}\label{cond psi}
\psi_{k,M}\alpha_{k,M'}=\alpha_{k,M},   \   \    \    \beta_{k,M}\psi_{k,M}=\beta_{k,M'}.
\end{equation}
If $i\neq k$ we set  $\psi_{i,M}=\ident_{M_i}$, and we denote $\psi_M=(\psi_{i,M})_{i\in Q_0}$. From the expression \eqref{ident M'} we observe  that  have the following filtration
\begin{equation}
\left\{0\right\}\subseteq \ker(\beta_{k,M})\subseteq \ima(\alpha_{k,M})\subseteq  M_{k}
\end{equation}
 
It can be checked that we can choose sections
\begin{eqnarray}\label{sec para psi}
\begin{array}{lcr}
\sigma_{1,M}:\ima(\alpha_{k,M})/\ker(\beta_{k,M})\longrightarrow \ima(\alpha_{k,M}),\\
\sigma_{2,M}:M_k/\ima(\alpha_{k,M})\longrightarrow M_k;
\end{array}
\end{eqnarray}
with fulfill 
\begin{equation}\label{cond de sigma_i}
\ima(\sigma_{1,M})=\alpha_{k,M}(\ker(\overline{\rho})), \ \
\ima(\beta_{k,M}\sigma_{2,M})=\ima(\overline{\sigma})
\end{equation}
If we define $\psi_{k,M}=(-\iota , -\iota \sigma_{1,M} , -\iota\sigma_{2,M})$, then  it can be proved that $\psi$ is a isomorphism.  By multiplying the respective matrix and taking into account \eqref{sec para psi} we get \eqref{cond psi}.

Since $f\psi_M-\psi_N\mu_k^-\mu_k^+(f)$ is confined to $k$, we have the commutative diagram in $\Lambda\modulo / [\add  S_k]$.
\begin{eqnarray}\label{psi natural}
\xymatrix{\mu_k^-\mu_k^+(M)\ar@{->}[r]^(0.6){\psi_{M}}\ar@{->}[d]_{\mu_k^-\mu_k^+(f)}& M\ar@{->}[d]^{f}\\
\mu_k^-\mu_k^+(N)\ar@{->}[r]_(0.6){\psi_{N}}& N
}
\end{eqnarray}
The other composition $\mu_k^+\mu_k^-$ is similar. Therefore the result follows.
\end{proof}


\begin{thebibliography}{AA}
\bibitem{APR} M. Auslander, M. I. Platzeck and I. Reiten. \emph{Coxeter functors without diagrams}. Trans. Amer. Math. Soc.250, PP. 1–46, 1979.

\bibitem{ASS} I. Assem, D. Simson and A. Skowro\'nski. \emph{Elements of the representation theory of associative algebras, Vol. 1, techniques of representation theory}. London Mathematical Society Student Text 65, Cambridge University Press, Cambridge, 2006. 

\bibitem{BGP} J. Bernstein, I. Gelfand, and V. Ponomarev.\emph{ Coxeter functors and Gabriel's theorem}.
Uspehi Mat. Nauk 28, no. 2 (170):19-33, 1973. Translated in Russian Math. Surveys
28,17-32, 1973.

\bibitem{BIRS} A. Buan, O. Iyama, I. Reiten, and D. Smith. \emph{Mutation of cluster-tilting objects and
potentials}. American J. Math, 133, no. 4:835-887, 2011.

\bibitem{DWZI} H. Derksen, J. Weyman and A. Zelevinsky. \emph{Quivers with potentials and their representations {I}: Mutations.} Selecta Math, 14, no. 1, 59 - 119,  2008. 

\bibitem{DWZII} H. Derksen, J. Weyman and A. Zelevinsky. \emph{Quivers with potentials and their representations {II}: Applications to cluster algebras.} J. Amer. Math. Soc, 23, no. 3, 749 - 790,  2010. 

\bibitem{DR} V. Dlab and C. M. Ringel. \emph{Representations of graphs and algebras,} Mem. Amer. Math. Soc., No. 173, 1976. 

\bibitem{FZI} S. Fomin and A. Zelevinsky. \emph{Cluster Algebras I: Foundations.} J. Amer. Math.
Soc, 15, 497 - 529, 2002.

\bibitem{Rin} C. M. Ringel. \emph{Some Remarks Concerning Tilting Modules and Tilted Algebras.
Origin. Relevance. Future.} Cambridge University Press, LMS Lecture Notes Series
332 edition. An appendix to the Handbook of Tilting Theory, 2007.
\end{thebibliography}
\end{document}